\documentclass{article}
\usepackage[utf8]{inputenc}
\usepackage{pdfpages}
\usepackage{url}
\usepackage{euler} 
\usepackage[toc,page]{appendix}
\usepackage{amssymb, amsthm,amsmath}
\usepackage[Conny]{fncychap}
\usepackage{graphicx}
\usepackage{adjustbox}
\usepackage{comment}
\usepackage{etex}
\usepackage{mathrsfs}
\usepackage[english]{babel}

\usepackage[scaled=0.85]{beramono}%% mono
\usepackage{tikz-cd}
\usepackage{float}
\usepackage{fancyhdr}

\usepackage[left=0cm,right=0cm,top=2cm,bottom=2cm,margin=3cm]{geometry}
\usepackage{amsthm}
\usepackage{bigints}
\usepackage{tikz}
\usepackage[linktoc=all]{hyperref}
\usepackage[capitalise]{cleveref}
\hypersetup{colorlinks=true,linkcolor=blue,citecolor=red}
\usetikzlibrary{matrix,arrows,decorations.pathmorphing}
\usetikzlibrary{patterns,calc,angles,quotes}
\usepackage{tikz-cd}
\usepackage[utf8]{inputenc}
\tikzset{commutative diagrams/.cd}
\newtheorem{theorem}{Theorem}[subsection]

\newtheorem{corollary}[theorem]{Corollary}

\newtheorem{lemma}[theorem]{Lemma}

\newtheorem{proposition}[theorem]{Proposition}

\theoremstyle{definition}
\newtheorem{definition}[theorem]{Definition}

\newtheorem{question}[theorem]{Question} 

\newtheorem{claim}[theorem]{Claim}

\newtheorem{example}[theorem]{Example}

\newtheorem{remark}[theorem]{Remark}

\newtheorem{notation}[theorem]{Notation}

\newcommand{\E}{\mathcal{E}}
\newcommand{\D}{\mathcal{D}}

\newcommand{\N}{\mathcal{N}}

\newcommand{\op}[1]{\operatorname{#1}}

\newcommand{\dd}{\delta}

\newcommand{\Ca}{\mathcal{C}}

\newcommand{\bb}{\bullet}

\numberwithin{subsection}{section}

\newcommand{\Set}{\op{Sets}}
\newcommand{\J}{\mathcal{J}}

\newcommand{\M}{\mathcal{M}}

\newcommand{\sset}{\op{Set}_{\Delta}}

\fancypagestyle{main}{\fancyhf{}
	\fancyhead[LE,RO]{\scriptsize Con. fun. using cat of simplices.}
	\fancyhead[RE,LO]{\scriptsize \rightmark }
	\fancyfoot[CE,CO]{\thepage}}

\title{Six-Functor Formalisms I : Constructing functors using category of simplices.}
\author{Chirantan Chowdhury}
\date{\today}

\begin{document}

\maketitle{}

\begin{abstract}
    This article is first in a series of papers where we  reprove the statements in constructing the Enhanced Operation Map and the abstract six-functor formalism developed by Liu-Zheng. In this paper, we prove a theorem regarding constructing functors between simplicial sets using the category of simplices. We shall reprove the statement using the language of marked simplicial sets and studying injective model structure on functor categories. The theorem is a crucial tool and will be used repeatedly in reproving the $\infty$-categorical compactification and construction of abstract six-functor formalisms in the forthcoming articles. 
\end{abstract}
\tableofcontents
\section{Introduction}

The six-functor formalism was developed by Grothendieck and others in order to understand duality statements. The initial setup was developed in the setting of triangulated categories. In the recent years with the development of $\infty$-categories due to Lurie (\cite{HTT},\cite{HA} and \cite{SAG}), six-functor formalism are formulated in the setting of stable presentable $\infty$-categories. This was developed by Liu and Zheng (\cite{Gluerestnerv} and \cite{liu2017enhanced}) in order to extend six operations for derived categories of $\ell$-adic sheaves from schemes to algebraic stacks. Since then, there have been many developments of six functor formalism in many contexts in algebraic, arithmetic geometry  and motivic homotopy theory(\cite{gulotta2022enhanced}, \cite{mann2022padic}, \cite{scholzesixfun}, \cite{khan2021generalized} and \cite{Chow1}). Unfortunately, the papers \cite{Gluerestnerv} and \cite{liu2017enhanced} are not published for quite some time and a lot of recent developments are based on crucial results developed in these papers. The main goal of the series of articles (this article followed by \cite{chowdhury2024sixfunctorformalismsii} and \cite{chowdhury2024sixfunctorformalismsiiiconstruction}) is to reprove a simplified version of main results in these papers assuming the language of higher categories developed by Lurie.\\

The two papers written by Liu and Zheng involves simplicial tools to extend the six functor formalism from smaller to bigger categories (schemes to algebraic stacks for example).  The crucial groundbreaking idea due to them is constructing the \textit{Enhanced Operation Map} which is a morphism between simplicial sets encoding all the six functors and important properties like base change, projection formula. They also developed an algorithm called the DESCENT program which is a method to extend the formalism from a smaller subcategory to the bigger category using the Enhanced Operation Map. Recently, Mann's thesis (\cite{mann2022padic}) uses a simpler convenient  notion of a \textit{3-functor formalism}  and proves analogous results of the DESCENT program in order to construct six functor formalism for $p$-torsion etale sheaves associated to a small $v$-stack. Although the notion of $3$-functor formalism is much easier to explain and understand, but the construction of such formalisms still uses crucial theorems proved by Liu and Zheng which are very technical and hard to understand.  Liu and Zheng prove a technical statement (\cite[Proposition 2.1.4]{Gluerestnerv}) which is used repeatedly to construct the Enhanced operation map.  In this article, we reprove a simpler version of this technical statement (\cref{mainthm}) which is sufficient enough for studying the six functor formalism. Let us motivate this technical statement by revisiting the extraordinary pushforward map in the context of 
 \'etale cohomology of schemes.\\

Let $f: X \to Y$ be a separated morphism of finite type of quasi-compact and quasi-separated schemes and $\Lambda$ be a torsion ring. We have the extraordinary pushforward map  \[ f_!: D(X,\Lambda) \to D(Y,\Lambda)  \]
	which when restricted to open immersions is the map $f_{\#}$ and to proper morphisms the map $f_*$ (here $D(X,\Lambda)$ is the derived category of $\Lambda$-constructible \'etale sheaves on $X$). The construction of $f_!$ involves the general theory of gluing two psuedofunctors developed by Deligne (\cite[Section 3]{Delignecohomologyproper}). Let us briefly recall the setup of the construction. 
 \begin{definition}\label{compactification2catdef}
      For any morphism $f$ as above, we consider the $2$-category of compactifications $\op{Sch}^{\op{comp}} $ whose objects are schemes and morphisms are triangles
	\begin{equation}
		\begin{tikzcd}
			{} & \overline{Y} \arrow[dr,"p"] & {} \\
			X \arrow[ur,hookrightarrow,"j"] && Y
		\end{tikzcd}
	\end{equation}
	where $j$ is open and $p$ is proper.
 \end{definition}
 It is important to note that one can compose morphisms of such form due to Nagata's theorem of compactification. Then one can define a pseudo-functor $F_c: \op{Sch}^{\op{comp}} \to \op{Cat}_1$ which sends a scheme $X$ to $D(X,\Lambda)$ and a triangle of the form above to the composition $p_* \circ j_{\#}$ (here $\op{Cat}_1$ denotes the $2$-category of categories). The theory of gluing in $2$-categories tell us that the functor $F_c$ can be extended to a functor $f_!$ from the category $\op{Sch}'$ consisting of schemes where morphisms are separated and finite type. In other words, the diagram 
	\begin{equation}
		\begin{tikzcd}
			\op{Sch}^{\op{comp}} \arrow[r,"F_c"] \arrow[d,"\op{pr}"] & \op{Cat}_1 \\
			\op{Sch}' \arrow[ur,"F_!"] & {}
		\end{tikzcd}
	\end{equation} 
	commutes. . \\
 The main points involved in the above construction are :
 \begin{enumerate}\label{conditionsforextensioninschemes}
     \item For any morphism $f \in \op{Sch}'$, the category of compactifications $\op{Comp}(f)$ (the fiber of $f$ over $\op{pr}$)  is filtered. In other words given any compactifications, there is a third compactifications which maps to both of the compactifications.
     \item Any two compactifications gives the equivalent functors between the derived categories. In other words for two compactifications $(j,p)$ and $(j',p')$ we have $p_* \circ j_{\#} \cong p'_*\circ j'_{\#}$. Thus we have a functor between categories:
     \begin{equation}\label{compactificationnattransform}
         \op{Comp}(f) \to \op{Fun}^{\cong}([1],\op{Cat}_1)
     \end{equation}
     where $\op{Fun}^{\cong}([1],\op{Cat}_1)$ is the largest groupoid inside $\op{Fun}([1],\op{Cat}_1)$.
 \end{enumerate}
In the setting of $\infty$-categories, we replace the category $\op{Sch}^{\op{comp}}$ by a simplicial set $\dd^*_2 N(\op{Sch}')^{cart}_{P,O}$. The $n$-simplices of $\dd^*_2N(\op{Sch}')^{\op{cart}}_{P,O}$ are $n \times n$ grids of the form 
	
	\begin{equation}
		\begin{tikzcd}
			X_{00} \arrow[r] \arrow[d] & X_{01} \arrow[r] \arrow[d] & \cdots & X_{0n} \arrow[d] \\
			X_{10} \arrow[r]  & X_{11} \arrow[r] & \cdots & X_{1n} \\
			\vdots & \vdots & \vdots & \vdots \\
			X_{n1} \arrow[r] & X_{n2} & \cdots & X_{nn}
		\end{tikzcd}
	\end{equation}
	where vertical arrows are open, horizontal arrows are proper and each square is a pullback square. Also one has a natural morphism $p_{\op{Sch}'}:\dd^*_2N(\op{Sch}')^{\op{cart}}_{P,O} \to N(\op{Sch}')$ induced by composition along the diagonal. The following result which is proven in the next article of the series (\cite{chowdhury2024sixfunctorformalismsii}) is the following:
 \begin{theorem}\cite[Corollary of Theorem 1.0.3]{chowdhury2024sixfunctorformalismsii}
     Let $F: \dd^*_2N(\op{Sch'})^{\op{cart}}_{P,O} \to \D$ be a functor where $\D$ is an $\infty$-category. Then there exists a functor $F': N(\op{Sch}') \to \D$ such that the diagram 
     \begin{equation}
         \begin{tikzcd}
             \dd^*_2N(\op{Sch'})^{\op{cart}}_{P,O} \arrow[r,"F"] \arrow[d,"p_{\op{Sch'}}"] & \D  \\
             N(\op{Sch}') \arrow[ur,"F'",swap] & {}
         \end{tikzcd}
     \end{equation}
     commutes.
     
 \end{theorem}
 The above theorem and the discussion above on the level of $2$-categories poses the following question :
 \begin{question}
 Given a morphism of simplicial sets $i: K' \to K$ and a morphism $f': K \to \Ca$ when $\Ca$ is an $\infty$-category. Under what conditions we can construct a morphism $f: K \to \Ca$ such that 
 \begin{equation}
     \begin{tikzcd}
         K' \arrow[r,"f'"] \arrow[d,"i"] & \Ca \\
         K \arrow[ur,"f",swap] & {}
     \end{tikzcd}
     \end{equation}
     commutes?
 \end{question}

Let us try to reformulate the conditions for gluing on the level of $2$-categories in the setting of simplicial sets:

\begin{enumerate}
    \item For every morphism $f$, there exists a filtered category $\op{Comp}(f)$. In the setting of simplicial sets, this must generalize to associating a weakly contractible simplicial set (a filtered category is weakly contractible) to every $\sigma :\Delta^n \to K$. This leads to motivating the use of \textit{the category of simplices}. The category of simplicies consists of objects which are pairs $(n,\sigma)$ where $n \ge 0$ and $ \sigma : \Delta^n \to K$ ( see \cref{catofsimpldef} for precise definition). The above condition can be reformulated into the existence of a functor
    \begin{equation}
        \N: (\Delta_{/K})^{op} \to \sset    \end{equation}
        such that $\N(n,\sigma)$ is weakly contractible.

      \item   The functor \cref{compactificationnattransform} can be realized as morphism in the functor category $\op{Fun}((\Delta_{/K})^{op},\sset)$. The left hand side of the functor is replaced by the functor $\N$ in the above paragraph. A generalization of $\op{Fun}^{\cong}([1],\D)$ is the \textit{mapping functor} (\cref{mappingfunctor})
      \begin{equation}
          \op{Map}[K,\Ca] : (\Delta_{/K})^{op} \to \sset ~ ~ (n,\sigma) \to \op{Fun}^{\cong}(\Delta^n,\Ca).
      \end{equation}
      Then the simplicial analogue of \cref{compactificationnattransform} is the existence of a natural transformation
      \begin{equation}
          \alpha : \N \to \op{Map}[K,\Ca].
      \end{equation}
\end{enumerate}

  We need some more notations in order to state the theorem. For any functor $F \in \op{Fun}((\Delta_{/K})^{op},\sset)$, let $\Gamma(F) = \op{lim}_{(\Delta_{/K})^{op}} F \in \sset$ (\cref{globalsectionfunctor}) and $i^*F$ to be the functor $(\Delta_{/K'})^{op} \to (\Delta_{/K})^{op} \xrightarrow{F} \sset$ (\cref{commutativityofglobalsecunderpullback}). The theorem is  stated as follows:
 \begin{theorem}(\cref{mainthm})
Let $K',K$ be simplicial sets and $\Ca$ be a $\infty$-category. Let $f': K' \to \Ca$ and $i:K' \to K$ be morphisms of simplicial sets. Let $\N \in (\sset)^{(\Delta_{/K})^{op}}$ and $\alpha: \N \to \op{Map}[K,\Ca]$ be a natural transformation. If
\begin{enumerate}
 \item (\textit{Weakly contractibility}) for $(n,\sigma) \in \Delta_{/K}$, $\N(n,\sigma)$ is weakly contractible,
 \item (\textit{Compatability with $f'$}) there exists $\omega \in \Gamma(i^*\N)_0$ such that $\Gamma(i^*\alpha)(\omega)=f'$ 
 \end{enumerate}
 
then there exists a map $f: K \to \Ca$ such that the following diagram 
\begin{equation}
    \begin{tikzcd}
        K' \arrow[r,"f'"] \arrow[d,"i"] & \Ca \\
        K \arrow[ur,"f"] & {}
    \end{tikzcd}
\end{equation}
 commutes. In other words, $f' \cong f \circ i $ in $\op{Fun}(K',\Ca)$.
\end{theorem}
The above theorem describes some conditions in order to ensure the construction of the functor $f$. In the theorem of $\infty$-categorical compactification (\cite[Theorem 1.0.3]{chowdhury2024sixfunctorformalismsii}) and partial adjoints (\cite{chowdhury2024sixfunctorformalismsiiiconstruction}), we shall verify these conditions in order to construct functors. \\

We briefly describe the sections of the article:
\begin{enumerate}
    \item In Section 2, we recall relevant notions of marked simplicial sets due to Lurie (\cite[Section 3.1]{HTT}). In particular we review the Cartesian model structure on marked simplicial sets.
    \item In Section 3, we firstly study the injective model structure on the functor category $\op{Fun}(\J,\sset)$ where $\J$ is an ordinary category.  We review the category of simplices and the mapping functor (\cref{mappingfunctor}). The section ends with proving an important proposition showing that the mapping functor is a fibrant object in the injective model structure (\cref{mappingfunctorinjfibrant}). This proposition helps us to prove \cref{mainthm} in the following section.
    \item In Section 4, we state and proof the theorem (\cref{mainthm}). We also see how gluing in two categories in the context of extraordinary pushforward map can be reproved using \cref{mainthm}.
    \item Appendices A and B concern on proving lemmas which are used in proving \cref{mappingfunctorinjfibrant}.
\end{enumerate}
 	\subsection*{Acknowledgments}
		 The paper was written while the author was a Post Doc at the University of Duisburg-Essen, supported by the ERC Grant "Quadratic Refinements in Algebraic Geometry". He would like to thank Alessandro D'Angelo for helpful discussions regarding the technical lemmas involved in the article. He would also like to express his gratitude to all members of ESAGA group in Essen. 

\section{Marked Simplicial Sets and Cartesian Model Structure}
In this section, we are going to recall definitions and results on marked simplicial sets (\cite[Section 3.1]{HTT}). Marked simplicial sets are simplicial sets with additional "markings". These play an important role in constructing the Straightening-Unstraightening functor (the $\infty$-categorical formalism of Grothendieck's construction of fibered categories). We shall also review the Cartesian model structure on marked simplicial sets which is needed for proving the Grothendieck's construction in $\infty$-categorical setting.

\subsection{Marked Simplicial Sets}
\begin{definition}\cite[Definition 3.1.0.1]{HTT}
A \textit{marked simplicial set} is a pair $(X,\E)$ where $X$ is a simplicial set and $\E$ is a set of edges $X$ which contains every degenerate edge. We will say that an edge of $X$ will be called marked if it belongs to $\E$.\\

A morphism $f : (X,\E) \to (X',\E')$ of marked simplicial sets is a map of simplicial sets $f: X \to X'$ having the property that $f(\E) \subset \E'$, The category of marked simplicial sets will be denoted by $\op{Set}^{+}_{\Delta}$.
\end{definition}
Below are the typical notations used in the category of marked simplicial sets.
\begin{notation}
\begin{enumerate}
    \item Let $S$ be simplicial set. Then let $S^{\flat} $ denote the marked simplicial set $S,s_0(S_0))$. Thus $S^{\flat}$ is the simplicial set $S$ with the markings containing only the degenerate edges of $S$.
    \item For a simplicial set $S$ let $S^{\sharp}$ denote the marked simplicial set $ (S,S_1)$, Thus $S^{\sharp}$ is the marked simplicial set where every edge is marked. 
    \item For a Cartesian fibration $p: X \to S$, let $X^{\natural}$ denote the marked simplicial set $(X,\E)$ where $\E$ is the set of $p$-Cartesian edges.
\end{enumerate}
\end{notation}

\begin{remark}
For an $\infty$-category $\Ca$, the unique map $p_{\Ca}:\Ca \to \Delta_0$ is a Cartesian fibration where the $p_C$-Cartesian edges are the equivalences in $\Ca$. Thus $\Ca^{\natural}$ is the marked simplical set $(\Ca,\E_{\Ca})$ where $\E_{\Ca}$ is the set of equivalences in $\Ca$.  
\end{remark}

The category of simplicial sets is Cartesian closed (\cite[Section 3.1.3]{HTT}). In other words, it has an internal mapping object which is defined as follows:

\begin{definition}
    Let $X,Y \in \op{Set}^{+}_{\Delta}$, then there exists an internal mapping object $Y^X \in \sset^+$ equipped with an "evaluation map" (a morphism of marked simplicial sets) \[ e_{X,Y}:Y^X \times X \to Y\] such that for every $Z \in \sset^+$, the map $e_{X,Y}$ induces bijections
    \[  \op{Hom}_{\sset^+}(Z,Y^X) \xrightarrow{e_{X,Y}} \op{Hom}_{\sset^+}(Z \times X,Y).\]
Given the internal mapping object $Y^X$, we can associate two simplicial sets $\op{Map}^{\flat}(X,Y)$ and $\op{Map}^{\sharp}(X,Y)$ which are defined follows:
\[ \op{Map}^{\flat}(X,Y)_n = \op{Hom}_{\sset^+}((\Delta^n)^{\flat} \times X, Y)\]
\[ \op{Map}^{\sharp}(X,Y)_n = \op{Hom}_{\sset^+}((\Delta^n)^{\sharp} \times X ,Y).\]
\end{definition}

\begin{remark}
    Some remarks on the internal mapping object are as follows:
    \begin{enumerate}
        \item By definition, we see that $\op{Map}^{\sharp}(X,Y) \subset \op{Map}^{\flat}(X,Y)$. 
        \item \label{markedsimpliskancomplex} For a marked simplicial set $X_{\E}=(X,\E)$ and $\infty$-category $\Ca$, $\op{Map}^{\flat}(X_{\E},\Ca^{\natural})$ is the full subcategory of the functor category $\op{Fun}(X,\Ca)$ spanned by objects $f: X \to \Ca$ which sends $\E_X$ to equivalences in $\Ca$.  Hence $\op{Map}^{\flat}(X_{\E},\Ca^{\natural})$ is an $\infty$-category.
        \item For a marked simplicial set $X_{\E}=(X,\E)$ and an $\infty$-category $\Ca$, the simplicial set $\op{Map}^{\sharp}(X,\Ca^{\natural})$ is the largest Kan complex contained in $\op{Map}^{\flat}(X_{\E},\Ca^{\natural})$. In case $X_{\E}= X^{\flat}$, then $\op{Map}^{\flat}(X^{\flat},\Ca^{\natural})$ is the largest Kan complex contained inside the $\infty$-categorry $\op{Fun}(X,\Ca)$.
        \item We have a pair of adjoint functors  : 
        \begin{equation}\label{flatforgetul} (-)^{\flat}:\sset \leftrightarrows \sset^+ : \op{forgetful} 
        \end{equation}
    \end{enumerate}
\end{remark}

\subsection{Cartesian Model Structure}
The category of marked simplicial sets admits a model structure called the Cartesian Model structure. 

\begin{definition}\cite[Proposition 3.1.3.7]{HTT}
    There exists a left proper combinatorial model structure on the category of marked simplicial sets $\sset^+$ called the \textit{Cartesian Model Structure} which is described as follows:
    \begin{enumerate}
        \item The cofibrations in $\sset^+$ are the morphisms $f : (X,\E_X) \to (Y,\E_Y)$ which are cofibrations when regarded as morphism of the underlying simplicial sets (i.e. monomorphisms).
        \item The weak equivalences are morphisms $f:X \to Y$ which are \textit{Cartesian equivalences} i.e. for every $\infty$-category $\Ca$ the natural morphism
        \begin{equation}
            \op{Map}^{\flat}(Y,\Ca^{\natural}) \to \op{Map}^{\flat}(X,\Ca^{\natural})
        \end{equation}
        is an equivalence of $\infty$-categories.
        \item The fibrations are those morphisms which have right lifting property with respect to every map which is a cofibration and a Cartesian equivalence.
    \end{enumerate}
\end{definition}
\begin{remark}\label{fibrantobjectsincartmodel}
    \item In the Cartesian Model structure, the fibrant objects are $\Ca^{\natural}$ where $\Ca$ is an $\infty$-category (\cite[Proposition 3.1.4.1]{HTT}).
\end{remark}
We have the notion of anodyne morphisms in category of simplicial sets which are exactly the trivial cofibrations in Kan model structure. In the context of marked simplicial sets, we have the notion of marked anodyne.
\begin{definition}\cite[Definition 3.1.1.1]{HTT}
    The class of \textit{marked anodyne morphisms} in $\sset^+$ is the smallest weakly saturated class of morphisms with the following properties:
    \begin{enumerate}
        \item For each $ 0<i<n$, the morphism $(\Lambda^n_i)^{\flat} \hookrightarrow (\Delta^n)^{\flat}$ is a marked anodyne.
        \item Let $\E$ denote the set of degenerate edges of $\Delta^n$ together with the final edge $\Delta^{\{n-1,n\}}$. Then for $n >0$, the morphism 
        \begin{equation}
            (\Lambda^n_n,((\Lambda^n_n)_1 \cap \E) \to (\Delta^n,\E)
        \end{equation}
        is a marked anodyne.
        \item The inclusion 
        \begin{equation}
            (\Lambda^2_1)^{\sharp} \coprod_{(\Lambda^2_1)^{\flat}} (\Delta^2)^{\flat} \to (\Delta^2)^{\sharp}
            \end{equation}
            is a marked anodyne.
        \item For every Kan complex $K$, the morphism $K^{\flat} \to K^{\sharp}$ is a marked anodyne.
    \end{enumerate}
\end{definition}

\begin{example}\label{markedanodyneexample}
    Marked anodynes are Cartesian equivalences (\cite[Remark 3.1.3.]{HTT}).
\end{example}

The following proposition describes the morphisms in $\sset^+$ which admits right lifting property along marked anodyne.
\begin{proposition}\cite[Proposition 3.1.1.6]{HTT}
A map $p: X \to S$ in $\sset^+$ admits right lifting property along marked anodyne morphisms iff the following conditions are satisfied :
\begin{enumerate}
    \item The map $p$ is an inner fibration of simplicial sets.
    \item An edge of $e$ is marked iff $p(e)$ is marked and $e$ is $p$-Cartesian.
    \item For every object $ y \in X$ and a marked edge $ \bar{e}:\bar{x} \to p(y)$ in $S$, there exists a marked edge $ e: x \to y$ such that $p(e)=\bar{e}$.
\end{enumerate}
In particular, $p:X \to S$ which are Cartesian fibrations and the marked edges of X are $p$-Cartesian edges satisfy these conditions above. 
\end{proposition}
An important of marked anodyne morphisms are the fact that they are stable under smash products along arbitrary cofibrations.
\begin{proposition}\cite[Proposition 3.1.2.3]{HTT}
    If $ f: X \to X'$ is a marked anodyne and $g: Y \to Y'$ is a cofibration in $\sset^+$, then the pushout product 
    \begin{equation}
        f~\square~g : X \times Y' \coprod_{X \times Y} X' \times Y \to X' \times Y'
    \end{equation}
    is a marked anodyne.
\end{proposition}

The following proposition gives a Quillen adjunction between model structures on simplicial sets and marked simplicial sets.

\begin{proposition}\label{quillenadjunctionsharpmap}
    The pair of adjoint functors :
\begin{equation}
     (-)^{\sharp} : \sset \leftrightarrows \sset^+ : G
\end{equation}
where  $G(Y,\E_Y)= Y_{\E_Y}$ the subsimplicial set of $Y$ spanned by edges in $\E_Y$ is  a Quillen adjunction between the Kan Model structure on $\sset$ and the Cartesian Model structue on $\sset^+$.
\end{proposition}
\begin{proof}
At first, we have the bijection of Hom sets
\begin{equation}
    \op{Hom}_{\sset^+}(X^{\sharp},(Y,\E_Y)) \cong \op{Hom}_{\sset}(X,Y_{\E_Y}).
\end{equation}
This is true because given a morphism $X \to Y_{\E_Y}$, it induces a morphism $X^{\sharp} 
\to (Y,\E_Y)$.\\
In order to show it is a Quillen adjunction, we show that $(-)^{\sharp}$ sends cofibrations to cofibrations and trivial cofibrations to trivial cofibrations. \\
Firstly, $(-)^{\sharp}$ sends cofibrations to cofibrations. This is true because the cofibrations are defined as cofibrations in the underlying map of simplicial sets. Thus we need to show that $p^{\sharp}: X^{\sharp} \to Y^{\sharp}$ sends anodyne morphisms to Cartesian equivalences (they are already cofibrations). In other words, we need to show that for an $\infty$-category $\Ca$, we need to show
\begin{equation}
    \op{Map}^{\flat}(Y^{\sharp},\Ca^{\natural}) \to \op{Map}^{\flat}(X^{\sharp},\Ca^{\natural})
\end{equation}
is a trivial Kan fibration. Unravelling the definition of Kan fibration, we need to show the existence of the dotted arrow:
\begin{equation}
    \begin{tikzcd}
        Y^{\sharp}\times (\partial\Delta^n)^{\flat} \coprod_{X^{\sharp}\times (\partial\Delta^n)^{\flat}} X^{\sharp}\times (\Delta^n)^{\flat} \arrow[d,hookrightarrow,"i"] \arrow[r,"\alpha"] & \Ca^{\natural} \\
        Y^{\sharp}\times(\Delta^n)^{\flat} \arrow[ur,dotted,"\alpha'"]& {}.
    \end{tikzcd}
\end{equation}
Let $K$ be the largest Kan complex inside $\Ca$. Then the morphism of $\alpha$ on the level of simplicial sets can be rewritten as 
\begin{equation}
                Y \times \partial\Delta^n \coprod_{X\times \partial\Delta^n} X\times \Delta^n\xrightarrow{\alpha}  K
\end{equation}

As the pushout product of an anodyne and a monomorphism of simplicial sets is anodyne, it admits lifting along the morphism $K \to \Delta^0$. Thus we have an extension
\begin{equation}
    Y \times \Delta^n \xrightarrow{\alpha'} K.
\end{equation}
Rewritting in terms of marked simplicial sets, we get a morphism \begin{equation}
    \alpha': Y^{\sharp}\times(\Delta^n)^{\flat} \to \Ca^{\natural}
\end{equation}
extending $\alpha$.\\
This completes the proof of the proposition.
\end{proof}
\section{Functors from  the Category of simplices}

In this section, we introduce the main tools needed for stating the main statement. Firstly, we introduce the Model structure on the category of functors from a small category $\J$ to $\sset$. Then we introduce the category of simplicies $\Delta_{/K}$ associated to a simplicial set $K$ and introduce the functor $\op{Map}[K,\Ca]$. The final goal of this section is to prove that functor $\op{Map}[K,\Ca]$ is fibrant in the model structure $(\sset)^{\Delta/K}$. We are mainly referring to  \cite[Section 2]{Gluerestnerv}

\subsection{Properties and model structure on $(\sset)^{\J}$.}

\begin{definition}
    Let $\J$ be a (small) ordinary category. Let $(\sset)^{\J}$ be the category where objects are functors from $\J \to \sset$ and morphisms are natural transformations.
\end{definition}

\begin{definition}
    As the Quillen model structure on $\sset$ is combinatorial and $\J$ is a small category, by \cite[Proposition A.2.8.2]{HTT}, there exists an  injective model structure on $(\sset)^{\J}$ which is defined by the following classes of morphisms:
    \begin{enumerate}
        \item A morphism $i: F \to G$ is called \textit{anodyne} if for every object $j \in J$, the morphism of simplicial sets $i(j): F(j) \to G(j)$ is anodyne in $\sset$.
        \item A morphism $i: F \to G$ is called a \textit{weak equivalence} if for every object $j \in J$, the morphism of simplicial sets $i(j) : F(j) \to G(j)$ is a weak equivalence in $\sset$.
        \item A morphism $i: F \to G$ is said to be \textit{injective fibration} if it has right lifting property with respect to anodyne morphisms. 
    \end{enumerate}
\end{definition}

\begin{notation}
    For every simplicial set $X$, we have the constant simplicial set functor $c(X):=X_{\J}$ defined by sending any object $j$ to the simplicial set $X$. The association is functorial and thus we have a functor :
     \[ c: \sset \to (\sset)^{\J}\]
    
\end{notation}

\begin{definition}\label{globalsectionfunctor}
We define the \textit{global section functor} 
\[ \Gamma : (\sset)^{\J} \to \sset\]
as follows : 
\[ \Gamma(F) =(\Gamma(F)_n := \op{Hom}_{(\sset)^{\J})}(\Delta^n_{\J},F))_n.\]
\end{definition}
\begin{example}
Let $F$ be the constant functor $c(X)$ where $X \in \sset$. Let us compute $\Gamma(F)$.
The $n$-simplices of $\Gamma(F)$ are given by the set of natural transformations from $\Delta^n_J \to c(X)$. Every such natural transformation is equivalent to give a single map $\Delta^n \to X$. In particular the $n$-simplices of $\Gamma(F)$ are given by $n$-simplices of $X$.Thus $\Gamma(F)= X$. In particular, we prove that $\Gamma \circ c = \op{id}_{\sset}$.

\end{example}
\begin{remark}
Recall from classical category theory, given a complete category $\Ca$ and an small category $I$, we have the pair of adjoint functors: 
 \[ c: \Ca \leftrightarrows \op{Fun}(I,\Ca) : \op{lim}\]
 where $\op{lim}$ is the functor which takes an object which is a functor $F : I \to \Ca$ to its limit $\op{lim}(F) \in \Ca$. \\

 Let $\Ca= \sset$ and $I=\J$. As the category of simplicial sets is complete,  we see that 
 \[ \Gamma = \op{lim}. \]
 In the other words, the global section functor is the limit functor which takes every functor to its limit in the category of simplicial sets.
 
\end{remark}

\begin{proposition}
The pair of functors: 
\[ c: (\sset)^{\J} \leftrightarrows \sset : \Gamma \]
is a Quillen adjunction with the injective model structure on $(\sset)^{\J}$ and the Quillen model structure on $\sset$.
\end{proposition}
\begin{proof}
We already know that $\Gamma$ and $c$ are adjoint to each other. In order to show it is a Quillen adjunction,we need to show that $c$ preserves cofibrations and trivial cofibrations. By definition $c$ takes an anodyne morphism $X \to S$ to $X_{\J} \to S_{\J}$ which is an anodyne in the injective model structure (as pointwise it is the anodyne morphism $X \to S$). It also preserves cofibrations in the similar argument.
\end{proof}
\begin{remark}\label{commutativityofglobalsecunderpullback}
    Let $g: \J \to \J'$ be a morphism of simplicial sets. This induces a pullback functor:
    on the level of functor categories:
    \[ g^* : (\sset)^{\J'} \to (\sset)^{\J}. \]
    Thus for any element $\M \in (\sset)^{\J'}$, we have a functor 
    \[ (-) \circ g:\Gamma(\M) \to \Gamma(g^*\M). \]
    Also for any morphism $ \psi: \N \to \M$ in $(\sset)^{\J'}$, we have the following commutative diagram of simplicial sets: 
    \begin{equation}
        \begin{tikzcd}
            \Gamma(\N) \arrow[r,"\Gamma(\psi)"] \arrow[d," (-) \circ g"] & \Gamma(\M) \arrow[d," (-)\circ g"] \\
            \Gamma(g^*\N) \arrow[r,"\Gamma(g^*\psi)",swap] & \Gamma(g^*\M).
        \end{tikzcd}
    \end{equation}
\end{remark}

\subsection{The category of simplices.}

\begin{definition}\label{catofsimpldef}
    Let $K$ be a simplicial set. Then the \textit{category of simplicies over $K$} is a category consisting of :
    \begin{enumerate}
        \item Objects : $(n,\sigma)$ where $n \ge 0$ and $\sigma \in K_n$.
        \item Morphisms: $ p:(n,\sigma)\to (m,\sigma')$ is a morphism $p: [n] \to [m]$ such that $p(\sigma)=\sigma'$.
    \end{enumerate}
\end{definition}

\begin{remark}
    The definition of the category of simplicies is functorial i.e. a morphism of simplicial sets $K \to K'$ induces a map of category of simplices \[ g^*: \Delta_{/K} \to \Delta_{/K'}. \]

\end{remark}

\begin{lemma}\label{colimitofcatofsimplices}
Let $K$ be a simplicial set. Then the colimit of the functor  
\[ \Delta_{/K} \to \sset\] given by \[ (n,\sigma) \to \Delta^n\] is given by $\Delta^n$.
\end{lemma}
\begin{proof}
    This follows from the fact that 
     $K$ is same as   $\op{colim}_{\sigma: \Delta^n \to K} \Delta^n$ (\cite[Lemma 3.1.3]{hoveymodel}). 
    
\end{proof}

\subsection{The mapping functor.}

\begin{definition}\cite[Notation 2.6]{Gluerestnerv}\label{mappingfunctor}
    Let $K$ be a simplicial set and $\Ca$ be a $\infty$-category. The \textit{mapping functor} \[ \op{Map}[K,\Ca] : (\Delta_{/K})^{op} \to \sset \] is defined as follows: 
    \[ (n,\sigma) \to \op{Map}^{\sharp}((\Delta^n)^{\flat},\Ca^{\natural}) \cong \op{Fun}^{\cong}(\Delta^n,\Ca). \]
 \end{definition}
\begin{remark}
    If $g: K \to K'$ a map of simplicial sets, then the pullback functor :

   \[ g^* : (\sset)^{(\Delta_{/K'})^{op}} \to (\sset)^{(\Delta_{/K})^{op}} \]
   maps $\op{Map}[K',\Ca] \to \op{Map}[K,\Ca]$.
\end{remark}

\begin{lemma}
    For a simplicial set $K$  and an $\infty$-category $\Ca$, we have the following equality of simplicial sets: 
     \[ \Gamma(\op{Map}[K,\Ca]) = \op{Map}^{\sharp}(K^{\flat},\Ca). \]
\end{lemma}
\begin{proof}
    By definition of $\Gamma$ , the set of $m$-simplices of $\Gamma(\op{Map}[K,\Ca])$ is the limit of the functor : 
    \[ \Delta_{/K} \to \Set~~;~~ (n,\sigma) \to \op{Hom}_{\sset^+}((\Delta^m)^{\sharp}\times (\Delta^n)^{\flat},\Ca^{\natural}). \] On the other hand, the $m$-simplices of $\op{Map}^{\sharp}(K^{\flat},\Ca^{\natural})$ is the set \[\op{Hom}_{\sset^+}( (\Delta^m)^{\sharp} \times K^{\flat},\Ca^{\natural}).\] \\
 Thus we need to show that \[ \op{lim}_{(n,\sigma)} \op{Hom}_{\sset^+}((\Delta^m)^{\sharp}\times (\Delta^n)^{\flat},\Ca^{\natural})= \op{Hom}_{\sset^+}((\Delta^m)^{\sharp} \times \op{colim}_{(n,\sigma)} (\Delta^n)^{\flat},\Ca^{\natural})\] is same as the set $\op{Hom}_{\sset^+}( (\Delta^m)^{\sharp} \times K^{\flat},\Ca^{\natural}).$

 This reduces down to showing that the colimit of the functor 
 \[ \Delta_{/K} \to \sset^+ ~~ (n,\sigma) \to (\Delta^n)^{\flat}\] is given by $K^{\flat}$. 
 By \cref{flatforgetul}, we know that $(-)^{\flat}$ preserves colimits. Thus we need to show that the colimit of the functor
 \[ \Delta_{/K} \to \sset ~~ (n,\sigma) \to (\Delta^n)\]  is $K$. This is precisely \cref{colimitofcatofsimplices}.
\end{proof}
\begin{remark}
   By above lemma, we see that the  global section of the functor $\op{Map}[K,\Ca]$ is given by $\op{Map}^{\sharp}(K^{\flat},\Ca^{\natural})$ which is the largest Kan complex contained in $\op{Fun}(K,\Ca)$. \\
   Objects in the connected components of this Kan complex yield us functors $K \to \Ca$ which are equivalent to one another. Thus in order to construct a functor $K \to \Ca$ upto equivalence, it is same to give an object in the connected part of $\Gamma(\op{Map}[K,\Ca])$.
\end{remark}
The following proposition is a crucial property of the functor $\op{Map}[K,\Ca]$.
\begin{proposition} \label{mappingfunctorinjfibrant}
    Let $K$ be a simplicial set and $\Ca$ be an $\infty$-category, the functor $\op{Map}[K,\Ca]$ is injectively fibrant with respect to the injective model structure on $(\sset)^{(\Delta_{/K})^{op}}$. In other words for every commutative diagram in $(\sset)^{(\Delta_{/K})^{op}}$ 
    \begin{equation}
        \begin{tikzcd}
        \N \arrow[r,"\alpha"] \arrow[d,"i"] & \op{Map}[K,\Ca] \\
        \M \arrow[ur,dotted,"\beta"] & {} 
        \end{tikzcd}
    \end{equation}
    where $i: \N \to \M$ is an anodyne, there exists a dotted arrow $\beta$ making the diagram commutative.
\end{proposition}
\begin{remark}
    Some remarks on the proposition are as follows:
    \begin{enumerate}
        \item If we only consider the functor from $\Delta^{\op{nd}}_{/K}$, then \cref{mappingfunctorinjfibrant} still holds. The proof gets much simpler as the technical part is to deal with the non-degenerate simplices.  However $\Delta^{nd}_{/K}$ is functorial with respect to the monomorphism of simplicial sets. This is generally not the case in the situation of compactifications and other cases. Thus we really need to prove it on the level of $\Delta_{/K}$,
        \item The proof that we give is rewriting the proof given by Liu-Zheng in \cite[Proposition 2.8]{Gluerestnerv}. We try to present a more cleaner version of the proof explaining the technicalities. 
    \end{enumerate}
\end{remark}
\begin{proof}
Let $\Delta^{\le n}_{/K}$ be the subcategory of $\Delta_{/K}$ spanned by $(m,\sigma)$ where $m \le n$. We construct the functor $\beta$ by induction on $n$. In other words, we construct $\beta_n$ from $\Delta_{/K}^{\le n}$. \\

\begin{enumerate}
    \item \textbf{The case $n=0$:} Let $(0,\sigma)$ be an element of $\Delta_{/K}$. We need to construct the map $\beta(0,\sigma)$ such that the following diagram of simplicial sets commute:
    \begin{equation}
        \begin{tikzcd}
            \N(0,\sigma) \arrow[d,"{i(0,\sigma)}"] \arrow[r,"{\alpha(0,\sigma)}"] & \op{Map}^{\sharp}((\Delta^0)^{\flat},\Ca^{\natural}) \\
            \M(0,\sigma) \arrow[ur,dotted,"{\beta(0,\sigma)}",swap] & {}.
            \end{tikzcd}
    \end{equation}
    As $\op{Map}^{\sharp}((\Delta^0)^{\flat},\Ca^{\natural})$ is the largest Kan complex contained inside $\Ca$. It admits right lifting property with respect to anodynes. As $i(0,\sigma)$ is anodyne, the statement holds. 
    \item \textbf{Induction step $n=k-1 \Rightarrow n=k:$} Let $(k,\sigma)$ be an object of $\Delta_{/K}$. We need to construct a map \begin{equation}
        \beta(k,\sigma): \M(k,\sigma) \to \op{Map}^{\sharp}((\Delta^k)^{\flat},\Ca^{\natural}) 
    \end{equation} 
    such that the map is compatible with the face maps $d^k_i: (k-1,\tau_i) \to (k,\sigma)$ and the degeneracy maps $s^{k-1}_j: (k,\sigma) \to (k-1,\rho_j)$. This means the above map $\beta(k,\sigma)$ must makes the following squares commute $\forall$ $ i,j$ :
    \begin{equation}
        \begin{tikzcd}
            \M(k-1,\rho_j) \arrow[r,"{\beta(k-1,\rho_j)}"] \arrow[d,"{\M(s^{k-1}_j)}",swap] & \op{Map}^{\sharp}((\Delta^{k-1})^{\flat},\Ca^{\natural}) \arrow[d,"{\op{Map}[K,\Ca](s^{k-1}_j)}"] \\
            \M(k,\sigma) \arrow[r,"{\beta(k,\sigma)}",swap] & \op{Map}^{\sharp}((\Delta^k)^{\flat},\Ca^{\sharp})
        \end{tikzcd}
    \end{equation}
 \begin{equation}
        \begin{tikzcd}
            \M(k,\sigma) \arrow[r,"{\beta(k,\sigma)}"] \arrow[d,"{\M(d^{k}_i)}",swap] & \op{Map}^{\sharp}((\Delta^{k})^{\flat},\Ca^{\natural}) \arrow[d,"{\op{Map}[K,\Ca](d^{k}_i)}"] \\
            \M(k-1,\tau_i) \arrow[r,"{\beta(k-1,\tau_i)}",swap] & \op{Map}^{\sharp}((\Delta^{k-1})^{\flat},\Ca^{\sharp})
        \end{tikzcd}
    \end{equation}

Let  \begin{equation}\label{deguniondef}
    \M^{\op{deg}}(k,\sigma) := \cup_j \M(s_j^{k-1})(\M(k-1,\rho_j)).
\end{equation}
The commutativity of the squares forces the following conditions :
\begin{enumerate}
    \item The first square gives that we have a map  \[ \beta(k,\sigma)^{\op{deg}} : \M(k,\sigma)^{\op{deg}}  \to \op{Map}^{\sharp}((\Delta^k)^{\flat},\Ca^{\natural}).\]
     In other words, we have a map \[ \beta(k,\sigma)^{\op{deg}} : (\M(k,\sigma)^{\op{deg}})^{\sharp} \times (\Delta^k)^{\flat} \to \Ca^{\natural}.\]
     \item The second square for all $i$, gives that combining $\beta(k-1,\tau_i)$, we have the map :
     \[ \beta(k,\sigma)^{\op{face}}: \M(k,\sigma) \to \op{Map}^{\sharp}((\partial\Delta^k)^{\flat},\Ca^{\natural}).\]
     In other words, we have a map of marked simplicial sets: 
     \begin{equation}
         \beta(k,\sigma)^{\op{face}} : \M(k,\sigma)^{\sharp} \times (\partial\Delta^k)^{\flat} \to \Ca^{\flat}.
     \end{equation}
\end{enumerate}
Note that the map $\beta(k,\sigma)$ can be written as a map of marked simplicial sets:
\[ \beta(k,\sigma) : (\M(k,\sigma))^{\sharp} \times (\Delta^k)^{\flat} \to \Ca^{\natural}.\]

The maps $\beta(k,\sigma)^{\op{deg}}, \beta(k,\sigma)^{\op{face}}$ along the existence of the map \[ \alpha(k,\sigma) : (\N(k,\sigma))^{\sharp} \times (\Delta^k)^{\flat} \to \Ca^{\natural}\]
 gives us the map 
 \begin{equation}
    \N(k,\sigma) \xrightarrow{\beta(k,\sigma)'} \Ca^{\flat}.
 \end{equation}
 where 
 \begin{equation}
     \M(k,\sigma)':= (\N(k,\sigma) \cup \M(k,\sigma)^{\op{deg}})^{\sharp} \times (\Delta^k)^{\flat} \coprod_{(\N(k,\sigma) \cup \M(k,\sigma)^{\op{deg}})^{\sharp} \times (\partial\Delta^k)^{\flat}} (\M(k,\sigma))^{\sharp} \times (\partial\Delta^k)^{\flat}.
 \end{equation}

Hence the construction  of $\beta(k,\sigma)$ is equivalent to existence of the dotted arrow 
\begin{equation}
    \begin{tikzcd}
        \M(k,\sigma)' \arrow[d,hookrightarrow,"j"] \arrow[r,"{\beta(k,\sigma)'}"] & \Ca^{\natural} \\
        \M(k,\sigma)^{\sharp} \times (\Delta^k)^{\flat} \arrow[ur,dotted,"{\beta(k,\sigma)}",swap]& {}
    \end{tikzcd}
\end{equation}
such that the diagram commutes. 
As $\Ca^{\natural}$ is fibrant in Cartesian model structure (\cref{fibrantobjectsincartmodel}), we need to show that $i$ is a trivial cofibration. \\
Notice that $j$ is pushout product of $j_1^{\sharp} : (\N(k,\sigma) \cup \M(k,\sigma)^{\op{deg}})^{\sharp} \hookrightarrow \M(k,\sigma)^{\sharp}$ and  a cofibration $j_2: (\partial\Delta^k)^{\flat} \hookrightarrow (\Delta^k)^{\flat}$. If we show that $j_1^{\sharp}$ is a trivial cofibration, then the pushout product of a trivial cofibration and a cofibration will be a trivial cofibration. Thus we show that $j_1^{\sharp}$ is a trivial cofibration in $\sset^+$. By \cref{quillenadjunctionsharpmap}, we see that it suffices to show that $j_1$ is an anodyne in $\sset$.

\begin{claim}
If $j_1':\N(k,\sigma)^{\op{deg}} \to \M(k,\sigma)^{\op{deg}}$ is anodyne, then $j_1$ is anodyne. 
\end{claim}

\begin{proof}
    We have the following commutative diagram 
    \begin{equation}
        \begin{tikzcd}
            \N(k,\sigma)^{\op{deg}} \arrow[r,hookrightarrow] \arrow[d,hookrightarrow,"j_1'"] & \N(k,\sigma) \arrow[d,"j_1''"] \arrow[dr,"{i(k,\sigma)}"] & {} \\
            \M(k,\sigma)^{\op{deg}} \arrow[r,hookrightarrow] & \M(k,\sigma)^{\op{deg}} \cup \N(k,\sigma) \arrow[r,hookrightarrow,"j_1"] & \M(k,\sigma)
             \end{tikzcd}
    \end{equation}
    where the square is a pushout square. Thus we have $j_1''$ is anodyne. As $j_1''$ and $i(k,\sigma)$ are anodyne and $j_1$ is a monomorphism, by \cref{anodyne2outof3}, we get that $j_1$ is anodyne.
\end{proof}

\begin{notation}
    Let $s:(k,\sigma) \to (k,\rho)$ be a degeneracy map. Then let $\M(k)_s$ denote the image of $\M(s)$.
\end{notation}
We prove the following claim :
\begin{claim}\label{distinctepiimageanodyne}
    Let $s_1,\cdots ,s_l$ be a collection of degeneracy maps where $s_j: (k,\sigma) \to (k-1,\rho_j)$. Then the morphism 
    \begin{equation}
        \cup_{j=1}^l \N(k)_{s_j} \to \cup_{j=1}^l \M(k)_{s_j}
     \end{equation}
     is anodyne.
\end{claim}
\textbf{Claim implies proof of $j_1'$ being anodyne:} Applying the above claim to $k$-distinct degeneracy maps gives us that $\N(k,\sigma)^{\op{deg}}:= \cup_{j=0}^{k-1} \N(\rho_j) \to \M(k,\sigma)^{\op{deg}}$ is anodyne hence completing the proof.

\begin{proof}[Proof of \cref{distinctepiimageanodyne}]
We prove this by induction on $l$. 
\begin{enumerate}
    \item \textbf{$l=1$:} In this case, we see that the  $\N(k)_{s_1} \cong \N(k-1,\sigma_1)$. This is true as $\N(s_1)$ is a monomorphism, then the image is isomorphic to its domain. Thus the map $\N(k)_{s_1} \to \M(k)_{s_1}$ is isomorphic to the map $\N(k-1,\rho_1) \to \M(k-1,\rho_1)$ which is an anodyne.
    \item  \textbf{$l=m-1 \implies l=m$:} For every $1 \le i \le m-1 $, we consider a pair of morphisms ($s_i,s_m$). By \cref{absolutepullbacklemma}, we get there exists an absolute pushout square :
    \begin{equation}
        \begin{tikzcd}
            (k,\sigma) \arrow[r,"s_i"] \arrow[d,"s_m"] & (k-1,\rho_i) \arrow[d] \\
            (k-1,\rho_m) \arrow[r,"s_i'"] & (k-2,\rho_{im}).
        \end{tikzcd}
    \end{equation}
    Thus for $\N$ (similarly for $\M$), we have a pullback square (as $\N$ and $\M$ are functors from $\Delta_{/K}^{op}$):
    \begin{equation}
    \begin{tikzcd}
            \N(k-2,\rho_{im}) \arrow[d] \arrow[r,"{\N(s_i')}"] &\N(k-1,\rho_m) \arrow[d,"{\N(s_m)}"]\\
        \N(k-1,\sigma_i) \arrow[r,"{\N(s_i)}"] & \N(k,\sigma).
    \end{tikzcd}
     \end{equation}
     By definition of pullback we get that the intersection of image of $\N(s_i')$ and $\N(s_m)$ is $\N(s_i's_m)$. This gives that we have a pushout square: 
     \begin{equation}
         \begin{tikzcd}
             \N(k)_{s_i's_m} \arrow[r] \arrow[d] & \N(k)_{s_1} \arrow[d] \\
             \N(k)_{s_m} \arrow[r] & \N(k)_{s_1}\cup\N(k)_{s_m}.
         \end{tikzcd}
     \end{equation}
    Combining the pushout squares for every $i$,we get the following pushout square
    \begin{equation}\label{pushoutsquareinduction}
        \begin{tikzcd}
            \cup_{i=1}^{m-1}\N(k)_{s_i's_m} \arrow[r]\arrow[d] & \cup_{i=1}^{m-1}\N(k)_{s_i} \arrow[d] \\
            \N(k)_{s_m} \arrow[r] & \cup_{i=1}^m \N(k)_{s_i}.
        \end{tikzcd}
    \end{equation}

Similar pushout square also holds for $\M$ too.\\
Consider the cube: 
 \begin{equation}
        \begin{tikzcd}
            \cup_{i=1}^{m-1}\N(k)_{s_i's_m} \arrow[rr] \arrow[dr,"f_0"] \arrow[dd] && \N(k)_{s_m} \arrow[dd] \arrow[dr,"f_1"] & {} \\
            {} & \cup_{i=1}^{m-1}\M(k)_{s_i's_m} \arrow[rr,"g_1"] \arrow[dd] && \M(k)_{s_m} \arrow[dd] \\
             \cup_{i=1}^{m-1}\N(k)_{s_i} \arrow[rr] \arrow[dr,"f_2"]  && \cup_{i=1}^m\N(k)_{s_i}  \arrow[dr,"f_3"] & {}\\
                   {} & \cup_{i=1}^{m-1}\M(k)_{s_i} \arrow[rr] && \cup_{i=1}^m\M(k)_{s_i}.       
        \end{tikzcd}
    \end{equation} 
In order to prove that $f_3$ is anodyne, we use \cref{anodyonepushoutlemma}. We verify the conditions in order to use \cref{anodyonepushoutlemma}:
\begin{enumerate}
    \item By induction $f_0,f_1$ and $f_2$ are anodynes.
    \item The front and back squares are pushout squares (\cref{pushoutsquareinduction}).
    \item We need to show that
    \begin{equation}
         f_{01}:\cup_{i=1}^{m-1} \M(k)_{s_i's_m}\coprod_{\cup_{i=1}^{m-1} \N(k)_{s_i's_m}} \N(k)_{s_m} \to \M(k)_{s_m}
    \end{equation}
    is a monomorphism. We show that coproduct is $\cup_{i=1}^{m-1} \M(k)_{s_i's_m} \bigcup \N(k)_{s_m}$. This shows that induced map $f_{01}$ is the union of $f_1$ and $g_1$ which are both monomorphisms. 

    In order to show that the coproduct is union, it is equivalent to show that for each $i$ $\M(k)_{s_i's_m} \cap \N(k)_{s_m} = \N(k)_{s_i's_m}$. \\

    For every  split epimorphism $s:(n,\zeta) \to (n',\zeta')$, we have the diagram :
     \begin{equation}
        \begin{tikzcd}
            \N(n',\zeta')\arrow[r,"{\N(s)}"] \arrow[d,"{i(n',\zeta')}"] & \N(n,\zeta) \arrow[r,"{\N(d)}"] \arrow[d,"{i(n,\zeta)}"] & \N(n',\zeta') \arrow[d,"{\i(n',\zeta')}"] \\
            \M(n'\zeta') \arrow[r,"{\M(s)}"] & \M(n,\zeta) \arrow[r,"{\M(d)}"] & \M(n,\zeta).
        \end{tikzcd}
    \end{equation}
    where $d$ is a section of $s$. As $i(n,\zeta)$ is a monomorphism, applying \cref{absolutepullbacklemma} gives us that left square is a pullback square. Applying this fact on the epimorphisms $s_m: (k,\sigma) \to (k-1,\rho_m)$ and $s_i':(k-1,\rho_m) \to (k-2,\rho_{im})$, we get the pullback squares :
     \begin{equation}
        \begin{tikzcd}
            \N(k-2,\rho_{im})\arrow[r,"{\N(s_i')}"] \arrow[d,"{i(k-2,\rho_{im})}"] & \N(k-1,\rho_m) \arrow[r,"\N(s_m)"] \arrow[d,"{i(k-1,\rho_m)}"] & \N(k,\sigma) \arrow[d,"{i(k,\sigma)}"] \\
            \M(k-2,\rho_{im}) \arrow[r,"{\M(s_i')}"] & \M(k-1,\rho_m) \arrow[r,"\M(s_m)"] & \M(k,\sigma).
        \end{tikzcd}
    \end{equation}
    This gives us a pullback square: 
    \begin{equation}
            \begin{tikzcd}
         \N(k-2,\rho_{im}) \arrow[r,"{\N(s_i')}"] \arrow[d,"{i(k-2,\rho_{im})}"] & \N(k-1,\rho_m) \arrow[d," {i(k,\sigma)\circ \N(s_m)}"] \\
         \M(k-2,\rho_{im}) \arrow[r,"{\N(s_ms_i')}"] & \M(k,\sigma).
    \end{tikzcd}
    \end{equation}
Taking intersection of the images of the maps in the pullback square we get that 
\begin{equation}
    \N(k)_{s_m}\cap \M(k)_{s_i's_m} = \N(s_i's_m).
\end{equation}

\end{enumerate}
Thus the conditions of \cref{anodyonepushoutlemma} are verified. Hence $f_3$ is anodyne. This completes the proof of \cref{distinctepiimageanodyne}.
    \end{enumerate}
\end{proof}
This also completes the induction of process for $n=k$. Hence we have constructed the map $\beta(k,\sigma)$ lifting $\alpha(k,\sigma)$.
    \end{enumerate}
\end{proof}
\section{The main theorem}
In the last section of the article, we prove the main theorem. 

\begin{theorem}\label{mainthm}
Let $K',K$ be simplicial sets and $\Ca$ be a $\infty$-category. Let $f': K' \to \Ca$ and $i:K' \to K$ be morphisms of simplicial sets. Let $\N \in (\sset)^{(\Delta_{/K})^{op}}$ and $\alpha: \N \to \op{Map}[K,\Ca]$ be a natural transformation. If
\begin{enumerate}
 \item (\textit{Weakly contractibility}) for $(n,\sigma) \in \Delta_{/K}$, $\N(n,\sigma)$ is weakly contractible,
 \item (\textit{Compatability with $f'$}) there exists $\omega \in \Gamma(i^*\N)_0$ such that $\Gamma(i^*\alpha)(\omega)=f'$,
 \end{enumerate}
 
then there exists a map $f: K \to \Ca$ such that the following diagram 
\begin{equation}
    \begin{tikzcd}
        K' \arrow[r,"f'"] \arrow[d,"i"] & \Ca \\
        K \arrow[ur,"f"] & {}
    \end{tikzcd}
\end{equation}
 commutes. In other words, $f' \cong f \circ i $ in $\op{Fun}(K',\Ca)$.
\end{theorem}
\begin{proof}
    Let \[ \N^{\triangleright} : \Delta_{/K}^{op}\to \sset ~~ (n,\sigma) \to \N(\sigma)^{\triangleright}.\]
    The canonical morphism  \[ \eta: \N \to \N^{\triangleright}\] is an anodyne. This is because for every weakly contractible simplicial set $K$, the inclusion $K \to K^{\triangleright}$ is an anodyne (\cite[Corollary 4.4.4.10]{HTT}). \\
    As $\op{Map}[K,\Ca]$ is injectively fibrant (\cref{mappingfunctorinjfibrant}), the map $\alpha$ has the following factorization :
\begin{equation}
    \begin{tikzcd}
        \N \arrow[r,"\alpha"]\arrow[d,"\eta"] & \op{Map}[K,\Ca] \\
        \N^{\triangleright} \arrow[ur,"\beta"] & {}.
    \end{tikzcd}
    \end{equation}
Let $z$ be the cone point of the simplicial set $\Gamma(\N^{\triangleright})$. Then we define \[f:= \Gamma(\beta)(z) :K \to \Ca. \]

We need to check is the compatibility of $f$ with $f'$. Pulling back along $i$, we get that $i^*\alpha$  has the following factorization: 
\begin{equation}
    \begin{tikzcd}
        i^*\N \arrow[r,"i^*\alpha"] \arrow [d,"i^*\eta"] & \i^*\op{Map}[K,\Ca]=\op{Map}[K',\Ca] \\
        (i^*\N)^{\triangleright} \arrow[ur,swap,"i^*\beta"] & {}.
    \end{tikzcd}
\end{equation}
Let $z'$ be the cone point of $\Gamma(i^*\N^{\triangleright})$. Note that we have the following commutative diagram of simplicial sets (\cref{commutativityofglobalsecunderpullback}): 
\begin{equation}
    \begin{tikzcd}
        \Gamma(\N^{\triangleright}) \arrow[r,"\Gamma(\beta)"] \arrow[d,"(-) \circ i"] & \Gamma(\op{Map}[K,\Ca]) = \op{Map}^{\sharp}(K^{\flat},\Ca^{\natural}) \arrow[d," (-) \circ i"]\\
        \Gamma(i^*\N^{\triangleright}) \arrow[r,"\Gamma(i^*\beta)"] & \Gamma(i^*\op{Map}[K,\Ca])= \op{Map}^{\sharp}(K^{\flat},\Ca^{\natural}).
    \end{tikzcd}
\end{equation}
Note that $ z \circ i =z'$. Thus we have that 
\begin{equation}\label{commutativityofconevertex}
    \Gamma(i^*\beta)(z')= \Gamma(i^*\beta)(z \circ i) = \Gamma(\beta)(z) \circ i = f \circ i .
\end{equation}
As the cone point is connected to every vertex, we see that $\Gamma(i^*\eta)(\omega)$ and $z'$ are connected. Thus 
\[ \Gamma(i^*\beta \circ i^*\eta)(\omega)= \Gamma(i^*\alpha)(\omega) =f' ~~ \text{and}~~ \Gamma(i^*\beta)(z')= \Gamma(\beta)(z) \circ i =f \circ i (\cref{commutativityofconevertex})\] are connected vertices in the Kan complex $\op{Map}^{\sharp}(K^{\flat},\Ca^{\natural})$. Thus we have 
\[ f' \equiv f \circ i ~~ \text{in}~~ \op{Fun}(K',\Ca). \]
    
    \end{proof}
\begin{remark}
    We point out some remarks of the theorem :
\begin{enumerate}
\item In the proof of the theorem, we used \cite[Corollary 4.4.4.10]{HTT}. The corollary says that let \begin{equation}
    h: S \to K
\end{equation}  be a map of simplicial sets where $S$ is a weakly contractible simplicial set and $K$ is a Kan complex. Then the morphism $h$ admits a colimit :
\begin{equation}
    h' : S^{\triangleright} \to K.
\end{equation}
    \item Let us explain how $f$ is evaluated for any arbitrary $n$-simplex $\sigma : \Delta^n \to K$. Evaluating $\alpha$ at the obejcts $(n,\sigma) \in \Delta_{/K}$, we get the following morphism :
    \begin{equation}
        \alpha(n,\sigma) : \N(n,\sigma) \to \op{Fun}^{\cong}(\Delta^n,\Ca).
    \end{equation}
    As $\N(n,\sigma)$ is weakly contractible, we have an extension :
    \begin{equation}
        \N(n,\sigma)^{\triangleright} \xrightarrow{\alpha'(n,\sigma)}\op{Fun}^{\cong}(\Delta^n,\Ca).
    \end{equation}
    We then define $f(\sigma)= \alpha'(n,\sigma)(z) = \op{lim} \alpha(n,\sigma)$ where $z$ is the cone point. \\
    If there exists $\sigma' \in K'_n$ mapping to $\sigma$, then the second condition of the theorem says that $f'(\sigma') \cong f(\sigma)$ in $\op{Fun}(\Delta^n,\Ca)$.
\end{enumerate}
\end{remark}
\begin{appendices}
\section{Remarks on absolute pullbacks.}

\begin{definition}
    Let $\Ca$ be an $1$-category. A square in $\Ca$ is said to be an \textit{absolute pullback (pushout)} if every functor $F: \Ca \to \D$ sends the square to a pullback (pushout) square in $\D$.
\end{definition}
\begin{lemma}\label{absolutepullbacklemma}
    Let $\Ca$ be a $1$-category. Suppose we have a diagram in $\Ca$ of the form
    \begin{equation}
        \begin{tikzcd}
            A \arrow[r,"d"] \arrow[d,"i"] & B \arrow[r,"s"] \arrow[d,"j"] & A \arrow[d,"i"] \\
            A' \arrow[r,"d'"] & B' \arrow[r,"s'"] & A'
        \end{tikzcd}
    \end{equation}
    where 
    \begin{enumerate}
        \item $s \circ d = \op{id}_A$,
        \item $s'\circ d' = \op{id}_{A'}$
        \item $j$ is a monomorphism (epimorphism).
    \end{enumerate}
    Then the left (right) square is a pullback (pushout). If $j$ is a split monomorphism (epimorphism), then the left (right) square is an absolute pullback (pushout).
    
\end{lemma}
\begin{proof}
    Let us prove that the left square is a pullback square. Let $X$ be an object in $\Ca$ and suppose we have a following diagram 
     \begin{equation}
        \begin{tikzcd}
        X \arrow[drr, bend left  = 80, "p"] \arrow[dr,dotted,"r"] \arrow[ddr,bend right= 80,"q"] & {} & {} & {}\\
          {}&  A \arrow[r,"d"] \arrow[d,"i"] & B \arrow[r,"s"] \arrow[d,"j"] & A \arrow[d,"i"] \\
            {}& A' \arrow[r,"d'"] & B' \arrow[r,"s'"] & A'.
        \end{tikzcd}
    \end{equation}
    such that $j \circ p =d' \circ q$.
    The aim is to construct a morphism $r$ such that 
    \begin{enumerate}
        \item $d \circ r =p$,
        \item $i \circ r= q$.
    \end{enumerate}
    Define $r:=s \circ p$. Let us verify the conditions mentioned above.
    \begin{enumerate}
    \item \textbf{$i \circ r =q$} 
    \begin{align*}
         i \circ r= i \circ s \circ p \\
         = s' \circ j \circ p \\
         = s' \circ d' \circ q \\
         = q.
    \end{align*}
    \item \textbf{$d \circ r = p.$}\\
    As $j$ is a monomorphism it is enough to show $j \circ d \circ r = j \circ p$.

    \begin{align*}
        j \circ d \circ r = j \circ d \circ s \circ p\\
        = d' \circ i \circ s \circ p \\
        = d' \circ s' \circ j \circ p \\
        = d' \circ s' \circ d'\circ q \\
        = d' \circ q \\
        = j \circ p
         \end{align*}
    
    \end{enumerate}
    This proves why the first square is a pullback square. \\
    In order to prove it is an absolute pullback, let $F:\Ca \to \D$ be a functor. In order to prove that $F$ applied to first square is still a pullback square, we shall repeat the proof above. Let $X'$ be an object of $\D$ and suppose we have the following diagram:
       \begin{equation}
        \begin{tikzcd}
        X \arrow[drr, bend left  = 80, "p"] \arrow[dr,dotted,"r"] \arrow[ddr,bend right= 80,"q"] & {} & {} & {}\\
          {}&  F(A) \arrow[r,"F(d)"] \arrow[d,"F(i)"] & F(B) \arrow[r,"F(s)"] \arrow[d,"F(j)"] & F(A) \arrow[d,"F(i)"] \\
            {}& F(A') \arrow[r,"F(d)'"] & F(B') \arrow[r,"F(s)'"] & F(A').
        \end{tikzcd}
    \end{equation}
    such that $F(j)\circ p = F(d')\circ q$.
    We define $r:= F(d)\circ p$.\\
     In order to prove that $F(i) \circ r = p$, the argument is the same as explained before. For the second equality, the fact that $j$ is a split monomorphism implies there exists $j':B' \to B$ such $j' \circ j = \op{id}_B$. We get the similar equality $F(j) \circ F(d) \circ r = F(j) \circ p$. Precomposing with $F(j')$, we get that $F(d)\circ r = p$.
    
\end{proof}
The following lemma is about the existence of absolute pushouts in the category of simplices.
\begin{lemma}
    Let $K$ be a simplicial set and let $(n,\sigma) \in \Delta_{/K}$. Let $s^{n-1}_i:(n,\sigma) \to (n-1,\sigma_i)$ and $s^{n-1}_j: (n,\sigma_i) \to (n-1,\sigma_j)$ be two distinct epimorphisms induced by the degeneracy maps. Then the pushout of $s^{n-1}_i$ along $s^{n-1}_j$ exist in $\Delta_{/K}$. Morever, it is an absolute pushout.
\end{lemma}
\begin{proof}
    Without loss of generality assume $i < j$. 
    By the simplicial identities in $\Delta$ we have the following diagram 
       \begin{equation}
        \begin{tikzcd}
        (n-1,\sigma_i)\arrow[r,"d^n_i"] \arrow[d,"s^{n-2}_{j-1}"] & (n,\sigma) \arrow[r,"s^{n-1}_i"] \arrow[d,"s^{n-1}_j"] & (n-1,\sigma_i) \arrow[d,"s^{n-2}_{j-1}"] \\
             (n-2,\sigma_{ij}) \arrow[r,"d^{n-1}_i"] & (n-1,\sigma_{ij}) \arrow[r,"s^{n-2}_i"] & (n-2,\sigma_{ij})
        \end{tikzcd}
    \end{equation}
    where all the squares are commutative. Notice that we have simplicial identites $s^{n-1}_i\circ d^n_i = \op{id}_{(n-1,\sigma_i)}$ and $s^{n-2}_i\circ d^{n-1}_i = \op{id}_{(n-2,\sigma_{ij})}$ (\cite[Section 3.1]{hoveymodel}). As $s^n_j$ is a split epimorphism ($s^{n-1}_j\circ d^n_j = \op{id}_{(n-1,\sigma_j)}$), applying \cref{absolutepullbacklemma}, gives us that pushout exists and the pushout square is an absolute pushout.
\end{proof}
\section{Remarks on anodynes.}

\begin{lemma}\label{anodyne2outof3}
    Let 
    \begin{equation}
        \begin{tikzcd}
            X \arrow[rr,"f"] \arrow[dr,"h",swap] && Y \arrow[dl,"g"] \\
            {} & Z & {}
        \end{tikzcd}
    \end{equation}
    be a commutative diagram in $\sset$. If $f,h$ are anodynes and $g$ is a monomorphism of simplicial sets, then $g$ is anodyne.
\end{lemma}
\begin{proof}
    By two out of three property of weak equivalences, we get that $g$ is a weak equivalence. As $g$ is a monomorphism and hence a cofibration, $g$ is a trivial cofibration i.e. an anodyne.
\end{proof}
The following lemma will be used in proving \cref{mappingfunctorinjfibrant}.
\begin{lemma}\label{anodyonepushoutlemma}
    Let
    \begin{equation}\label{anodynecube}
        \begin{tikzcd}
            A_0 \arrow[rr] \arrow[dr,"f_0"] \arrow[dd] && A_1 \arrow[dd] \arrow[dr,"f_1"] & {} \\
            {} & B_0 \arrow[rr] \arrow[dd] && B_1 \arrow[dd] \\
             A_2 \arrow[rr] \arrow[dr,"f_2"]  && A_3  \arrow[dr,"f_3"] & {}\\
                   {} & B_2 \arrow[rr] && B_3       
        \end{tikzcd}
    \end{equation}
    be a cube in $\sset$ with the following properties :
    \begin{enumerate}
        \item $f_0,f_1$ and $f_2$ are anodynes.
        \item The front and back squares are pushout squares.
        \item The morphism $f_{01}: B_0 \coprod_{A_0}A_1 \to B_1$ is a monomorphism.
    \end{enumerate}
    Then $f_3$ is anodyne.
\end{lemma}
\begin{proof}
    Let $B_1':=A_1 \coprod_{A_0} B_0$. Consider the cube :
       \begin{equation} \label{anodynecubepushout}
        \begin{tikzcd}
            A_0 \arrow[rr] \arrow[dr,"f_0"] \arrow[dd] && A_1 \arrow[dd] \arrow[dr,"f_1'"] & {} \\
            {} & B_0 \arrow[rr] \arrow[dd] && B_1' \arrow[dd] \\
             A_2 \arrow[rr] \arrow[dr,"f_2"]  && A_3  \arrow[dr,"f_3'"] & {}\\
                   {} & B_2 \arrow[rr] && B_3'      
        \end{tikzcd}
    \end{equation}
    where all the six squares are pushouts. As anodynes are preserved under pushouts, we get that $f_1'$ and $f_3'$ are anodyne morphisms. The  cube \cref{anodynecubepushout} decomposes \cref{anodynecube} into a bigger diagram :
       \begin{equation}
        \begin{tikzcd}
            A_0 \arrow[rr] \arrow[dr,"f_0"] \arrow[dd] && A_1 \arrow[dd] \arrow[rr,"\op{id}"] \arrow[dr,"f_1'"] && A_1 \arrow[dd] \arrow[dr,"f_1"] \\
            {} & B_0 \arrow[rr] \arrow[dd] && B_1'\arrow[dd] \arrow[rr,"f_{01}",swap] && B_1 \arrow[dd] \\
             A_2 \arrow[rr] \arrow[dr,"f_2"]  && A_3  \arrow[dr,"f_3'"] \arrow[rr,"\op{id}",swap] &&  A_3 \arrow[dr,"f_3"] & {}\\
                {} & B_2 \arrow[rr] && B_3' \arrow[rr,"f_{23}"] && B_3 .
        \end{tikzcd}
    \end{equation}
    As $f_1'$ and $f_1$ are anodynes and $f_{01}$ is a monomorphism, by \cref{anodyne2outof3}, we get that $f_{01}$ is anodyne. As the squares in the front face of the bigger cubes are all pushouts, we get that $f_{23}$ is anodyne (being the pushout of $f_{01}$). Thus $f_3= f_{23} \circ f_3'$ is anodyne. 
\end{proof}

\end{appendices}

\bibliographystyle{abbrv}

\begin{thebibliography}{10}

\bibitem{Chow1}
C.~Chowdhury.
\newblock Motivic homotopy theory of algebraic stacks.
\newblock {\em Ann. K-Theory}, 9(1):1--22, 2024.

\bibitem{chowdhury2024sixfunctorformalismsii}
C.~Chowdhury.
\newblock {S}ix-{F}unctor {F}ormalisms {II} : {T}he $\infty$-categorical compactification.
\newblock \url{https://arxiv.org/abs/2412.03231}, 2024.

\bibitem{chowdhury2024sixfunctorformalismsiiiconstruction}
C.~Chowdhury.
\newblock {S}ix-{F}unctor {F}ormalisms {III}: {T}he construction and extension of {6FF}s.
\newblock \url{https://arxiv.org/abs/2412.20548}, 2024.

\bibitem{Delignecohomologyproper}
P.~Deligne.
\newblock Cohomologie a supports propres.
\newblock In {\em Th{\'e}orie des Topos et Cohomologie Etale des Sch{\'e}mas}, pages 250--480, Berlin, Heidelberg, 1973. Springer Berlin Heidelberg.

\bibitem{gulotta2022enhanced}
D.~Gulotta, D.~Hansen, and J.~Weinstein.
\newblock An enhanced six-functor formalism for diamonds and v-stacks.
\newblock \url{https://arxiv.org/abs/2202.12467}, 2022.

\bibitem{hoveymodel}
M.~Hovey.
\newblock {\em Model categories}, volume~63 of {\em Mathematical Surveys and Monographs}.
\newblock American Mathematical Society, Providence, RI, 1999.

\bibitem{khan2021generalized}
A.~A. Khan and C.~Ravi.
\newblock Generalized cohomology theories for algebraic stacks.
\newblock {\em Adv. Math.}, 458:Paper No. 109975, 104, 2024.

\bibitem{Gluerestnerv}
Y.~{Liu} and W.~{Zheng}.
\newblock {Gluing restricted nerves of $\infty$-categories}.
\newblock \url{https://arxiv.org/pdf/1211.5294.pdf}, Nov. 2012.

\bibitem{liu2017enhanced}
Y.~Liu and W.~Zheng.
\newblock {Enhanced six operations and base change theorem for higher Artin stacks}.
\newblock \url{https://arxiv.org/pdf/1211.5948.pdf}, Sept. 2017.

\bibitem{HTT}
J.~Lurie.
\newblock {\em Higher topos theory}, volume 170 of {\em Annals of Mathematics Studies}.
\newblock Princeton University Press, Princeton, NJ, 2009.

\bibitem{HA}
J.~Lurie.
\newblock Higher algebra.
\newblock \url{http://people.math.harvard.edu/~lurie/papers/HA.pdf}, 2017.

\bibitem{SAG}
J.~Lurie.
\newblock Spectral algebraic geometry.
\newblock \url{https://www.math.ias.edu/~lurie/papers/SAG-rootfile.pdf}, 2018.

\bibitem{mann2022padic}
L.~Mann.
\newblock A $p$-adic 6-functor formalism in rigid-analytic geometry.
\newblock \url{https://arxiv.org/abs/2206.02022}, 2022.

\bibitem{scholzesixfun}
P.~Scholze.
\newblock Six-functor formalisms.
\newblock \url{https://people.mpim-bonn.mpg.de/scholze/SixFunctors.pdf }.

\end{thebibliography}

\end{document}